\documentclass{amsart}%
\usepackage{amssymb}
\usepackage{amsmath}
\usepackage{graphicx}
\usepackage{amsfonts}
\usepackage{color}
\usepackage{graphicx}%
\newtheorem{theorem}{Theorem}
\theoremstyle{plain}

\newtheorem{definition}{Definition}
\newtheorem{example}{Example}

\newtheorem{proposition}{Proposition}

\numberwithin{equation}{section}

\begin{document}
\title{$A_{\infty}$-bialgebras of Type $\left(  m,n\right)  $}
\author{Ainhoa Berciano$^{1}$}
\address{Departamento de Did\'actica de la Matem\'atica y de las Ciencias Experimentales\\
Universidad del Pa\'{\i}s Vasco-Euskal Herriko Unibertsitatea\\
Ram\'{o}n y Cajal, 72.\\
48014-Bilbao(Bizkaia). SPAIN}
\email{ainhoa.berciano@ehu.es}
\author{Sean Evans}
\address{Department of Mathematics\\
University of Pittsburgh\\
Pittsburgh, PA 15260}
\email{sme26@pitt.edu}
\author{Ronald Umble$^{2}$}
\address{Department of Mathematics\\
Millersville University of Pennsylvania\\
Millersville, PA. 17551}
\email{ron.umble@millersville.edu}
\thanks{$^{1}$This work was partially supported by \textquotedblleft Computational
Topology and Applied Mathematics\textquotedblright\ PAICYT research project
FQM-296, Spanish MEC project MTM2006-03722}
\thanks{$^{2}$This research was funded in part by a Millersville University faculty
research grant}
\date{January 6, 2010}

\begin{abstract}
An $A_{\infty}$\emph{-bialgebra} \emph{of type} $\left(  m,n\right)  $ is a
Hopf algebra $H$ equipped with a \textquotedblleft
compatible\textquotedblright\ operation $\omega_{m}^{n}:H^{\otimes
m}\rightarrow H^{\otimes n}$ of positive degree. We determine the structure
relations for $A_{\infty}$-bialgebras of type $\left(  m,n\right)  $ and
construct a purely algebraic example for each $m\geq2$ and $m+n\geq4.$

\end{abstract}
\maketitle

\section{Introduction}

Let $\Lambda$ be a (graded or ungraded) commutative ring with unity, and let
$m,n\in\mathbb{N}$ with $m+n\geq4$. An $A_{\infty}$\emph{-bialgebra} \emph{of
type} $\left(  m,n\right)  $ is a graded\ $\Lambda$-Hopf algebra $H$ equipped
with a \textquotedblleft compatible\textquotedblright\ multilinear operation
$\omega_{m}^{n}\in Hom^{m+n-3}\left(  H^{\otimes m},H^{\otimes n}\right)  $.
The first examples of type $\left(  1,3\right)  $ were given by H.J. Baues in
\cite{Baues}. Recently, the first and third authors observed that when $p$ is
an odd prime and $n\geq3$, examples of type $\left(  1,p\right)  $ appear as a
pair of tensor factors $E\otimes\Gamma$ in the mod $p$ homology of an
Eilenberg-Mac Lane space $K(\mathbb{Z},n)$ \cite{ABU1}. Examples of type
$\left(  m,1\right)  $ and the first \textquotedblleft
non-operadic\textquotedblright\ example of type $\left(  2,2\right)  $ were
constructed by the third author in \cite{Umble}; the first example of type
$\left(  2,3\right)  ,$ which appears here as Example \ref{ex1}, was
constructed by the second author in his undergraduate thesis \cite{Evans}. In
this paper we determine the structure relations for $A_{\infty}$-bialgebras of
type $\left(  m,n\right)  $ and construct a purely algebraic example for each
$\left(  m,n\right)  $ with $m\geq2$ and $m+n\geq4.$

The paper is organized as follows: In Section 2 we construct the components of
the particular biderivative we need. In Section 3 we review M. Markl's
\textquotedblleft fraction product\textquotedblright\ defined in \cite{Markl}
and use it to compute the $A_{\infty}$-bialgebra structure relations. Finally,
as a consequence of Theorem \ref{thm1} in Section 4, we obtain the following
examples:\textit{ Let }$R$ \textit{be a commutative ring with unity, let
}$m,n\in\mathbb{N}$\textit{ with }$m\geq2$\textit{ and} $m+n\geq4,$\textit{
and let }$\Lambda=E\left(  x\right)  $\textit{ be an exterior }$R$%
\textit{-algebra generated by }$x$ \textit{of degree }$2m-3.$\textit{ Let
}$H=E\left(  y\right)  $\textit{ be an exterior }$\Lambda$\textit{-algebra
generated by }$y$\textit{ of degree }$1,$ \textit{let} $\mu$\textit{ and
}$\Delta$\textit{ denote the trivial product and primitive coproduct on }$H,$
\textit{and define }$\omega_{m}^{n}\in Hom^{m+n-3}\left(  H^{\otimes
m},H^{\otimes n}\right)  $\textit{ by }$\omega_{m}^{n}\left(  y|\cdots
|y\right)  =xy|y|\cdots|y.$\textit{ Then }$\left(  H,\mu,\Delta,\omega_{m}%
^{n}\right)  $\textit{ is an }$A_{\infty}$\textit{-bialgebra of type }$\left(
m,n\right)  .$

\section{The Biderivative}

Given a graded $R$-module $H$ of finite type, consider the bigraded module
$M=\left\{  M_{m}^{n}=Hom\left(  H^{\otimes m},H^{\otimes n}\right)  \right\}
_{m,n\geq1}$ and a family of $R$-multilinear operations
\[
\left\{  \omega_{m}^{n}\in Hom^{m+n-3}\left(  H^{\otimes m},H^{\otimes
n}\right)  \right\}  _{n,m\geq1}.
\]
The sum $\omega=\sum\omega_{m}^{n}$ extends uniquely to its
\emph{biderivative} $d_{\omega}\in TTM,$ constructed by Saneblidze and the
second author in \cite{SU3}, and there is a (non-bilinear) $\circledcirc
$-product on $TM\subset TTM,$ which extends Gerstenhaber's $\circ_{i}%
$-products on $M_{\ast}^{1}\oplus M_{1}^{\ast} .$ The family
$\left\{  \omega_{m}^{n}\right\}  $ defines an $A_{\infty}$-bialgebra
structure on $H$ whenever $d_{\omega}\circledcirc d_{\omega}=0$, and the
structure relations are the homogeneous components of the equation $d_{\omega
}\circledcirc d_{\omega}=0$.

Consider a (biassociative) Hopf algebra $\left(  H,\Delta,\mu\right)  $
together with an operation $\omega_{m}^{n}\in Hom^{m+n-3}\left(  H^{\otimes
m},H^{\otimes n}\right)  $ with $m+n\geq4.$ Let us construct the component of
the biderivative $d_{\omega}$ in $TM$ when $\omega=\Delta-\mu+\omega_{m}^{n}.$
Let $\uparrow$ and $\downarrow$ denote the suspension and desuspension
operators, which shift dimension $+1$ and $-1,$ respectively. Let $\bar
{H}=H/H_{0}.$ For purposes of computing signs, we assume that $H$ is
simply-connected; once the signs have been determined, the simple-connectivity
assumption will be dropped. The \emph{bar construction of }$H$ is the tensor
coalgebra $BA=T\left(  \uparrow\bar{H}\right)  $ with differential induced by
$\mu;$ the \emph{cobar construction of }$H$ is the tensor algebra $\Omega
H=T\left(  \downarrow\bar{H}\right)  $ with differential induced by $\Delta$
(for details see \cite{Mac Lane}). Given a map $f:M\rightarrow N$ of graded
$\Lambda$-module modules, let $f_{i,j}$ denote the map $\mathbf{1}^{\otimes
i}\otimes f\otimes\mathbf{1}^{\otimes j}:N^{\otimes i}\otimes M\otimes
N^{\otimes j}\rightarrow N^{\otimes i+j+1}.$

Freely extend the map $\downarrow^{\otimes2}\Delta\uparrow:\downarrow\bar
{H}\rightarrow\Omega H$ as a derivation of $\Omega H,$ and dually, cofreely
extend the map $-\uparrow\mu\downarrow^{\otimes2}:\uparrow^{\otimes2}\bar
{H}^{\otimes2}\rightarrow BH$ as a coderivation of $BH.$ These extensions
contribute the components
\[
\delta^{k}=\sum_{i=1}^{k}\left(  -1\right)  ^{i+1}\Delta_{i-1,k-i}\text{ \ and
\ }-\partial_{k}=\sum_{i=1}^{k}\left(  -1\right)  ^{i+1}\mu_{i-1,k-i}%
\]
to the biderivative, where $k\geq1$ and the signs are the Koszul signs that
appear when factoring out suspension and desuspension operators.

Next, cofreely extend the map $\uparrow\Delta\downarrow:\uparrow\bar
{H}\rightarrow B\left(  H^{\otimes2}\right)  $ as a coalgebra map; this
extension
\[
\sum_{k\geq1}\left(  \uparrow\Delta\downarrow\right)  ^{\otimes k}%
:BH\rightarrow B\left(  H^{\otimes2}\right)
\]
contributes the components%
\[
\sum_{k\geq1}\left(  -1\right)  ^{\left\lfloor k/2\right\rfloor }%
\Delta^{\otimes k}.
\]
Dually, freely extend the map $-\downarrow\mu\uparrow:\downarrow\bar
{H}^{\otimes2}\rightarrow\Omega H$ as an algebra map; this extension%
\[
\sum_{k\geq1}\left(  -\downarrow\mu\uparrow\right)  ^{\otimes k}:\Omega\left(
H^{\otimes2}\right)  \rightarrow\Omega H
\]
contributes the components%
\[
\sum_{k\geq1}\left(  -1\right)  ^{\left\lfloor \left(  k+1\right)
/2\right\rfloor }\mu^{\otimes k}.
\]

Additional terms arise from various extensions involving $\omega_{m}^{n}$.
When $n=2,$ the components%
\[
\omega_{m}^{2}-\omega_{m}^{2}\otimes\Delta+\left(  -1\right)  ^{m}%
\Delta\otimes\omega_{m}^{2}-\omega_{m}^{2}\otimes\omega_{m}^{2}+\cdots.
\]
arise from the cofree extension of $\uparrow\Delta\downarrow+\uparrow
\omega_{m}^{2}\downarrow^{\otimes m}:\uparrow\bar{H}\oplus\uparrow^{\otimes
m}\bar{H}^{\otimes m}\rightarrow B\left(  H^{\otimes2}\right)  $ as a
coalgebra map $BH\rightarrow B\left(  H^{\otimes2}\right)  ;$ and when $m=2,$
the components%
\[
\omega_{2}^{n}+\mu\otimes\omega_{2}^{n}-\left(  -1\right)  ^{n}\omega_{2}%
^{n}\otimes\mu-\omega_{2}^{n}\otimes\omega_{2}^{n}+\cdots
\]
arise from the free extension of $\downarrow^{\otimes n}\omega_{2}^{n}%
\uparrow-\downarrow\mu\uparrow:\downarrow\bar{H}^{\otimes2}\rightarrow\Omega
H$ as an algebra map $\Omega\left(  H^{\otimes2}\right)  \rightarrow\Omega H$.
In particular, when $m=n=2$, the component of $d_{\omega}$ in $TM$ is
\begin{align*}
d_{\omega}  &  =\delta^{\ast}-\partial_{\ast}+\omega_{2}^{2}-\omega_{2}%
^{2}\otimes\omega_{2}^{2}+\cdots+\sum_{k\geq2}\left[  \left(  -1\right)
^{\left\lfloor \left(  k+1\right)  /2\right\rfloor }\mu^{\otimes k}+\left(
-1\right)  ^{\left\lfloor k/2\right\rfloor }\Delta^{\otimes k}\right] \\
&  +\mu\otimes\omega_{2}^{2}-\omega_{2}^{2}\otimes\mu+\cdots-\omega_{2}%
^{2}\otimes\Delta+\Delta\otimes\omega_{2}^{2}+\cdots.
\end{align*}

When $m\neq2$ or $n\neq2,$ additional terms arise from extensions governed by
the following generalization of an $\left(  f,g\right)  $-(co)derivation:

\begin{definition}
Let $\left(  A,\mu_{A}\right)  $ and $\left(  B,\mu_{B}\right)  $ be graded
algebras and let $f,g:A\rightarrow B$ be algebra maps. A map $h:A\rightarrow
B$ of degree $p$ is an $\left(  f,g\right)  $\emph{-derivation of degree} $p$
if $h\mu_{A}=\mu_{B}\left(  f\otimes h+h\otimes g\right)  .$ There is the
completely dual notion of an $\left(  f,g\right)  $\emph{-coderivation of
degree} $p.$
\end{definition}

\noindent Thus a $\left(  \mathbf{1,1}\right)  $-(co)derivation of degree $1$
is a classical (co)derivation and an $\left(  f,g\right)  $-(co)derivation of
degree $1$ is a classical $\left(  f,g\right)  $-(co)derivation. Define
$f^{1}=g_{1}=\mathbf{1,}$ and for $m,n\geq2$ define%
\begin{align*}
f^{n}  &  =\left(  \Delta\otimes\mathbf{1}^{\otimes n-2}\right)  \cdots\left(
\Delta\otimes\mathbf{1}\right)  \Delta=\left(  \mathbf{1}^{\otimes n-2}%
\otimes\Delta\right)  \cdots\left(  \mathbf{1}\otimes\Delta\right)
\Delta\text{ and}\\
g_{m}  &  =\mu\left(  \mu\otimes\mathbf{1}\right)  \cdots\left(  \mu
\otimes\mathbf{1}^{\otimes m-2}\right)  =\mu\left(  \mathbf{1}\otimes
\mu\right)  \cdots\left(  \mathbf{1}^{\otimes m-2}\otimes\mu\right)  .
\end{align*}

If $n>2,$ cofreely extend $\left(  -1\right)  ^{\left\lfloor n/2\right\rfloor
+1}\uparrow f^{n}\downarrow:\uparrow\bar{H}\rightarrow B\left(  H^{\otimes
n}\right)  $ as a coalgebra map $BH\rightarrow B\left(  H^{\otimes n}\right)
$ and obtain%
\[
F^{n}=\sum_{k\geq1}\left(  -1\right)  ^{\left(  \left\lfloor n/2\right\rfloor
+1\right)  k}\left(  \uparrow f^{n}\downarrow\right)  ^{\otimes k};
\]
then cofreely extend $\uparrow\omega_{m}^{n}\downarrow^{\otimes m}%
:\uparrow^{\otimes m}\bar{H}^{\otimes m}\rightarrow B\left(  H^{\otimes
n}\right)  $ as an $\left(  F^{n},F^{n}\right)  $-coderivation $BH\rightarrow
B\left(  H^{\otimes n}\right)  $ of degree $n-2.$ This extension contributes
the components%
\[
\omega_{m}^{n}+\left(  -1\right)  ^{\left\lfloor \left(  n+1\right)
/2\right\rfloor }\left(  \omega_{m}^{n}\otimes f^{n}-\left(  -1\right)
^{m}f^{n}\otimes\omega_{m}^{n}\right)  +\cdots
\]
to the biderivative, where we have applied the identity $\left\lfloor
n/2\right\rfloor +n\equiv\left\lfloor \left(  n+1\right)  /2\right\rfloor $
$\left(  \operatorname{mod}2\right)  .$ In particular, when $m=2$ we have%
\begin{align*}
d_{\omega}  &  =\delta^{\ast}-\partial_{\ast}+\omega_{2}^{n}-\omega_{2}%
^{n}\otimes\omega_{2}^{n}+\cdots+\sum_{k\geq2}\left[  \left(  -1\right)
^{\left\lfloor \left(  k+1\right)  /2\right\rfloor }\mu^{\otimes k}+\left(
-1\right)  ^{\left\lfloor k/2\right\rfloor }\Delta^{\otimes k}\right] \\
&  +\left(  -1\right)  ^{\left\lfloor \left(  n+1\right)  /2\right\rfloor
}\left(  \omega_{2}^{n}\otimes f^{n}-f^{n}\otimes\omega_{2}^{n}\right)
+\cdots+\mu\otimes\omega_{2}^{n}-\left(  -1\right)  ^{n}\omega_{2}^{n}%
\otimes\mu+\cdots.
\end{align*}

Dually, if $m>2,$ freely extend $\left(  -1\right)  ^{\left\lfloor \left(
m+1\right)  /2\right\rfloor }\downarrow g_{m}\uparrow:\downarrow\bar
{H}^{\otimes m}\rightarrow\Omega H$ as an algebra map $\Omega\left(
H^{\otimes m}\right)  \rightarrow\Omega H$ and obtain%
\[
G_{m}=\sum_{k\geq1}\left(  -1\right)  ^{\left\lfloor \left(  m+1\right)
/2\right\rfloor k}\left(  \downarrow g_{m}\uparrow\right)  ^{\otimes k};
\]
then freely extend the map $\downarrow^{\otimes n}\omega_{m}^{n}%
\uparrow:\downarrow\bar{H}^{\otimes m}\rightarrow\Omega H$ as a $\left(
G_{m},G_{m}\right)  $-derivation $\Omega\left(  H^{\otimes m}\right)
\rightarrow\Omega H$ of degree $m-2.$ This extension contributes the
components
\[
\omega_{m}^{n}-\left(  -1\right)  ^{\left\lfloor m/2\right\rfloor }\left(
g_{m}\otimes\omega_{m}^{n}-\left(  -1\right)  ^{n}\omega_{m}^{n}\otimes
g_{m}\right)  +\cdots.
\]
In particular, when $n=2$ we have%
\begin{align*}
d_{\omega}  &  =\delta^{\ast}-\partial_{\ast}+\omega_{m}^{2}-\omega_{m}%
^{2}\otimes\omega_{m}^{2}+\cdots+\sum_{k\geq2}\left[  \left(  -1\right)
^{\left\lfloor \left(  k+1\right)  /2\right\rfloor }\mu^{\otimes k}+\left(
-1\right)  ^{\left\lfloor k/2\right\rfloor }\Delta^{\otimes k}\right] \\
&  -\omega_{m}^{2}\otimes\Delta+\left(  -1\right)  ^{m}\Delta\otimes\omega
_{m}^{2}+\cdots-\left(  -1\right)  ^{\left\lfloor m/2\right\rfloor }\left(
g_{m}\otimes\omega_{m}^{2}-\omega_{m}^{2}\otimes g_{m}\right)  +\cdots.
\end{align*}
And finally, if $m,n\neq2$ we have%
\begin{align*}
d_{\omega}  &  =\delta^{\ast}-\partial_{\ast}+\omega_{m}^{n}+\sum_{k\geq
2}\left[  \left(  -1\right)  ^{\left\lfloor \left(  k+1\right)
/2\right\rfloor }\mu^{\otimes k}+\left(  -1\right)  ^{\left\lfloor
k/2\right\rfloor }\Delta^{\otimes k}\right] \\
&  +\left(  -1\right)  ^{\left\lfloor \left(  n+1\right)  /2\right\rfloor
}\left(  \omega_{m}^{n}\otimes f^{n}-\left(  -1\right)  ^{m}f^{n}\otimes
\omega_{m}^{n}\right)  +\cdots\\
&  -\left(  -1\right)  ^{\left\lfloor m/2\right\rfloor }\left(  g_{m}%
\otimes\omega_{m}^{n}-\left(  -1\right)  ^{n}\omega_{m}^{n}\otimes
g_{m}\right)  +\cdots.
\end{align*}
This completes the construction of the component of $d_{\omega}$ in $TM.$

To determine the $A_{\infty}$-bialgebra structure relations, we must define
the $\circledcirc$-product and extract the homogeneous components of the
equation $d_{\omega}\circledcirc d_{\omega}=0.$

\section{The $\circledcirc$-Product}

As mentioned above, a family of operations $\omega=\left\{  \omega_{m}%
^{n}\right\}  $ defines an $A_{\infty}$-bialgebra structure on $H$ whenever
$d_{\omega}\circledcirc d_{\omega}=0$, and the structure relations appear as
homogeneous components of the equation $d_{\omega}\circledcirc d_{\omega}=0$.
The $\circledcirc$-product is a restriction of the \emph{fraction product
}$\diagup:TM\times TM\rightarrow TM$\ defined by M. Markl in \cite{Markl} and
reviewed below. For $\alpha,\beta\in TM$ define%
\[
\alpha\circledcirc\beta=\left\{
\begin{array}
[c]{cl}%
\alpha\diagup\beta, & \text{\textit{if} }\alpha\text{ \textit{and} }%
\beta\text{ \textit{are components of} }d_{\omega}\\
0, & \text{\textit{otherwise.}}%
\end{array}
\right.
\]

Consider the \emph{submodule\ }$S$ \emph{of special elements} in the free
$\operatorname*{PROP}$ $M$ whose additive generators are either
indecomposables $\theta_{m}^{n}\in M_{m}^{n}$ of dimension $m+n-3$ or
monomials $\alpha_{\mathbf{x}}^{\mathbf{y}}\in TM\times TM$ expressed as
\textquotedblleft elementary fractions\textquotedblright\ of the form%
\begin{equation}
\alpha_{\mathbf{x}}^{\mathbf{y}}=\frac{\alpha_{p}^{y_{1}}\cdots\alpha
_{p}^{y_{q}}}{\alpha_{x_{1}}^{q}\cdots\alpha_{x_{p}}^{q}}, \label{fraction}%
\end{equation}
where $\alpha_{x_{i}}^{q}$ and $\alpha_{p}^{y_{j}}$ are additive generators of
$S$ and the $j^{th}$ output of $\alpha_{x_{i}}^{q}$ is linked to the $i^{th}$
input of $\alpha_{p}^{y_{j}}.$ All elementary fractions define the fraction product.

In terms of a composition, let $\sigma_{q,p}:\left(  H^{\otimes q}\right)
^{\otimes p}\overset{\approx}{\rightarrow}\left(  H^{\otimes p}\right)
^{\otimes q}$ be the standard permutation of tensor factors and note that
$\alpha_{p}^{y_{1}}\cdots\alpha_{p}^{y_{q}}\in Hom\left(  \left(  H^{\otimes
p}\right)  ^{\otimes q},H^{\otimes\Sigma y_{i}}\right)  $ and $\alpha_{x_{1}%
}^{q}\cdots\alpha_{x_{p}}^{q}\in Hom\left(  H^{\otimes\Sigma x_{j}},\left(
H^{\otimes q}\right)  ^{\otimes p}\right)  ;$ then $\alpha_{\mathbf{x}%
}^{\mathbf{y}}=\alpha_{p}^{y_{1}}\cdots\alpha_{p}^{y_{q}}$ $\circ$
$\sigma_{q,p}$ $\circ$ $\alpha_{x_{1}}^{q}\cdots\alpha_{x_{p}}^{q}\in
Hom\left(  H^{\otimes\Sigma x_{j}},H^{\otimes\Sigma y_{i}}\right)  .$ Whereas
juxtaposition in the \textquotedblleft numerator\textquotedblright\ and
\textquotedblleft denominator\textquotedblright\ denotes tensor product in
$TM$, $\dim\alpha_{\mathbf{x}}^{\mathbf{y}}=\sum_{i,j}\dim\alpha_{x_{i}}%
^{q}+\dim\alpha_{p}^{y_{j}}$.

We picture indecomposables and fraction products two ways. First, an
indecomposable generator $\theta_{m}^{n}$ is pictured as a double corolla with
$m$ inputs and $n$ outputs as in Figure 1, and a general elementary fraction
$\alpha_{\mathbf{x}}^{\mathbf{y}}$ is pictured as a connected non-planar graph
as in Figure 2. \
\[
\hspace*{-0.25in}%
\begin{array}
[c]{c}%
\theta_{m}^{n}=
\end{array}%
\begin{array}
[c]{c}%
n\\%
{\includegraphics[
height=0.4445in,
width=0.4514in
]%
{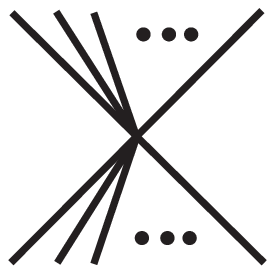}%
}
\\
m
\end{array}
\]

\begin{center}
Figure 1. An indecomposable pictured as a double corolla.\bigskip%

\[
\alpha_{111}^{12}\text{ \ }=\text{ \ }\frac{\alpha_{3}^{1}\alpha_{12}^{11}%
}{\alpha_{1}^{2}\alpha_{1}^{2}\alpha_{1}^{2}}\text{ \ }=\text{ \
\raisebox{-0.2517in}{\includegraphics[
trim=0.000000in 0.126900in 0.000000in 0.000000in,
height=0.7092in,
width=0.7904in
]%
{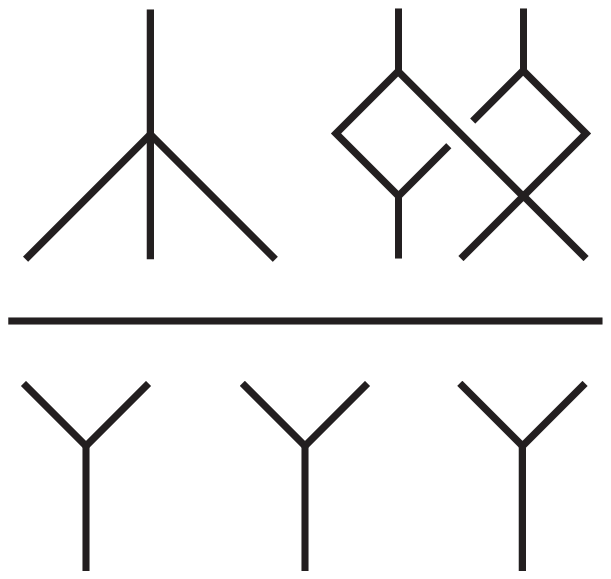}%
}
\ }=\text{ \
\raisebox{-0.2863in}{\includegraphics[
height=0.7524in,
width=0.6054in
]%
{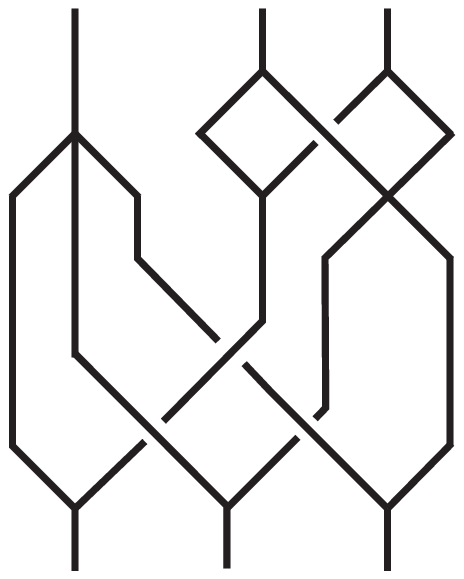}%
}
}%
\]

Figure 2. A fraction product pictured as a non-planar graph.\bigskip
\end{center}

Second, identify the pair of isomorphic modules $\left(  H^{\otimes q}\right)
^{\otimes p}\approx\left(  H^{\otimes p}\right)  ^{\otimes q}$ with the point
$\left(  p,q\right)  \in\mathbb{N}^{2}$ and think of indecomposables and
fraction products as operators on $\mathbb{N}^{2}$. An indecomposible
$\omega_{m}^{n}$ is pictured as a \textquotedblleft
transgressive\ arrow\textquotedblright\ $\left(  m,1\right)  \longrightarrow
\left(  1,n\right)  $ and a fraction product $\left(  \alpha_{p}^{y_{1}}%
\cdots\alpha_{p}^{y_{q}}\right)  \diagup\left(  \alpha_{x_{1}}^{q}\cdots
\alpha_{x_{p}}^{q}\right)  $ as a \textquotedblleft transgressive
path\textquotedblright\ $\left(  \Sigma x_{i},1\right)  \longrightarrow\left(
q,p\right)  \longrightarrow\left(  1,\Sigma y_{j}\right)  $ (see in Figure 3).
While this representation is helpful conceptually, it is not faithful. For
example, both $\alpha_{21}^{1}=\omega_{2}^{1}\diagup\left(  \omega_{2}%
^{1}\otimes\mathbf{1}\right)  $ and $\alpha_{12}^{1}=\omega_{2}^{1}%
\diagup\left(  \mathbf{1}\otimes\omega_{2}^{1}\right)  $ are represented by
the path $\left(  3,1\right)  \rightarrow\left(  2,1\right)  \rightarrow
\left(  1,1\right)  $.%

\begin{center}
\includegraphics[
trim=0.998137in 0.000000in 0.000000in 0.000000in,
height=1.9476in,
width=6.1376in
]%
{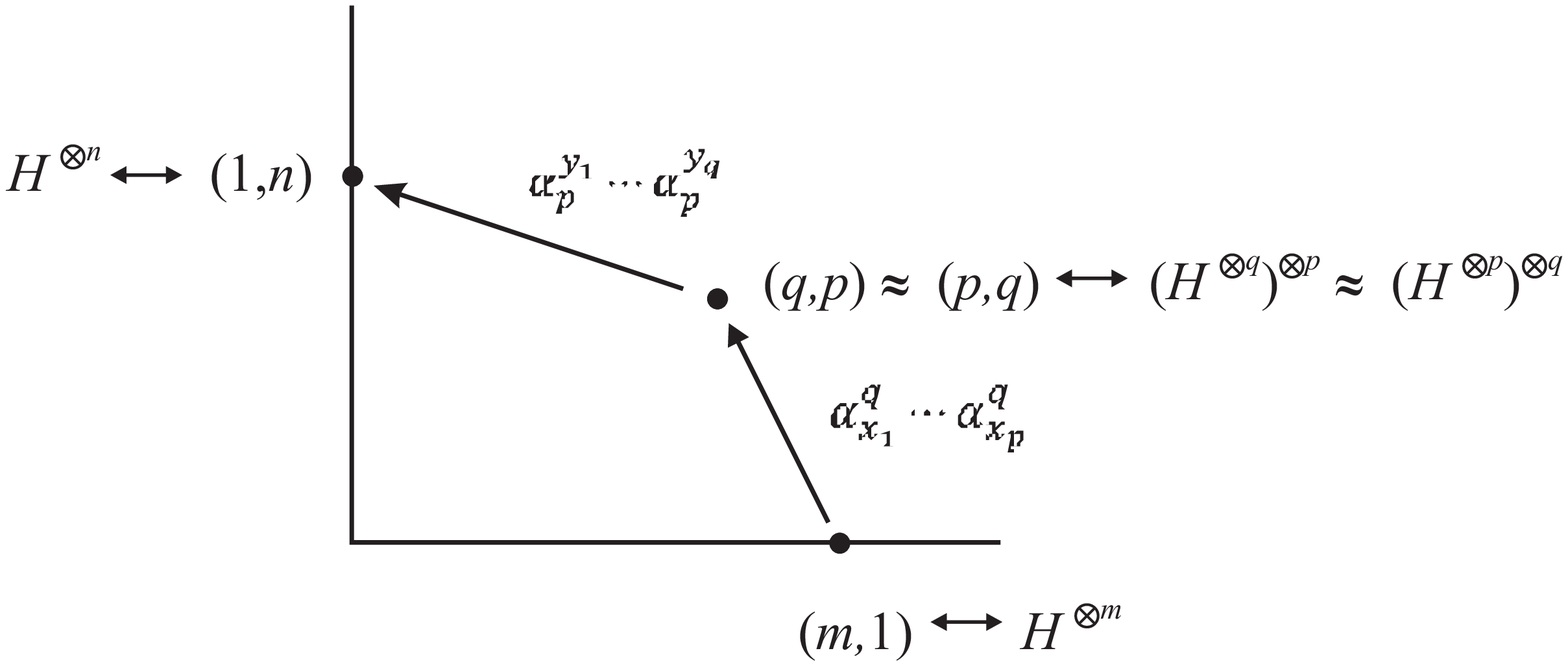}%
\\
{}%
\end{center}

\begin{center}
Figure 3. A fraction product as a composition of operators on $\mathbb{N}%
^{2}.\bigskip$
\end{center}

Up to sign, the homogeneous component of $d_{\omega}\circledcirc d_{\omega}$
with $m$ inputs and $n$ outputs consists of $\omega_{m}^{n}$ plus the sum of
all fractions $\alpha\circledcirc\beta$ thought of as paths $\left(
m,1\right)  \overset{\beta}{\rightarrow}\left(  q,p\right)  \overset{\alpha
}{\rightarrow}\left(  1,n\right)  .$

\section{$A_{\infty}$-Bialgebras of Type $\left(  m,n\right)  $}

Everything we need to determine the desired $A_{\infty}$-bialgebra structure
relations is now in place. In addition to the (co)associativity relations
$\mu\left(  \mu\otimes\mathbf{1}\right)  =\mu\left(  \mathbf{1}\otimes
\mu\right)  $ and $\left(  \Delta\otimes\mathbf{1}\right)  \Delta=\left(
\mathbf{1}\otimes\Delta\right)  \Delta,$ and the classical Hopf relation
$\Delta\mu=\left(  \mu\otimes\mu\right)  \sigma_{2,2}\left(  \Delta
\otimes\Delta\right)  $, the homogeneous components of $d_{\omega}\circledcirc
d_{\omega}=0$ with $m+1$ inputs and $n$ outputs produce Structure Relation 1:%
\[
\delta^{n}\omega_{m}^{n}=\left(  g_{m}\otimes\omega_{m}^{n}-\left(  -1\right)
^{n}\omega_{m}^{n}\otimes g_{m}\right)  \sigma_{2,m}\Delta^{\otimes m}.
\]%
\begin{center}
\includegraphics[
trim=0.000000in -0.125584in 0.000000in 0.000000in,
height=2.8435in,
width=3.3814in
]%
{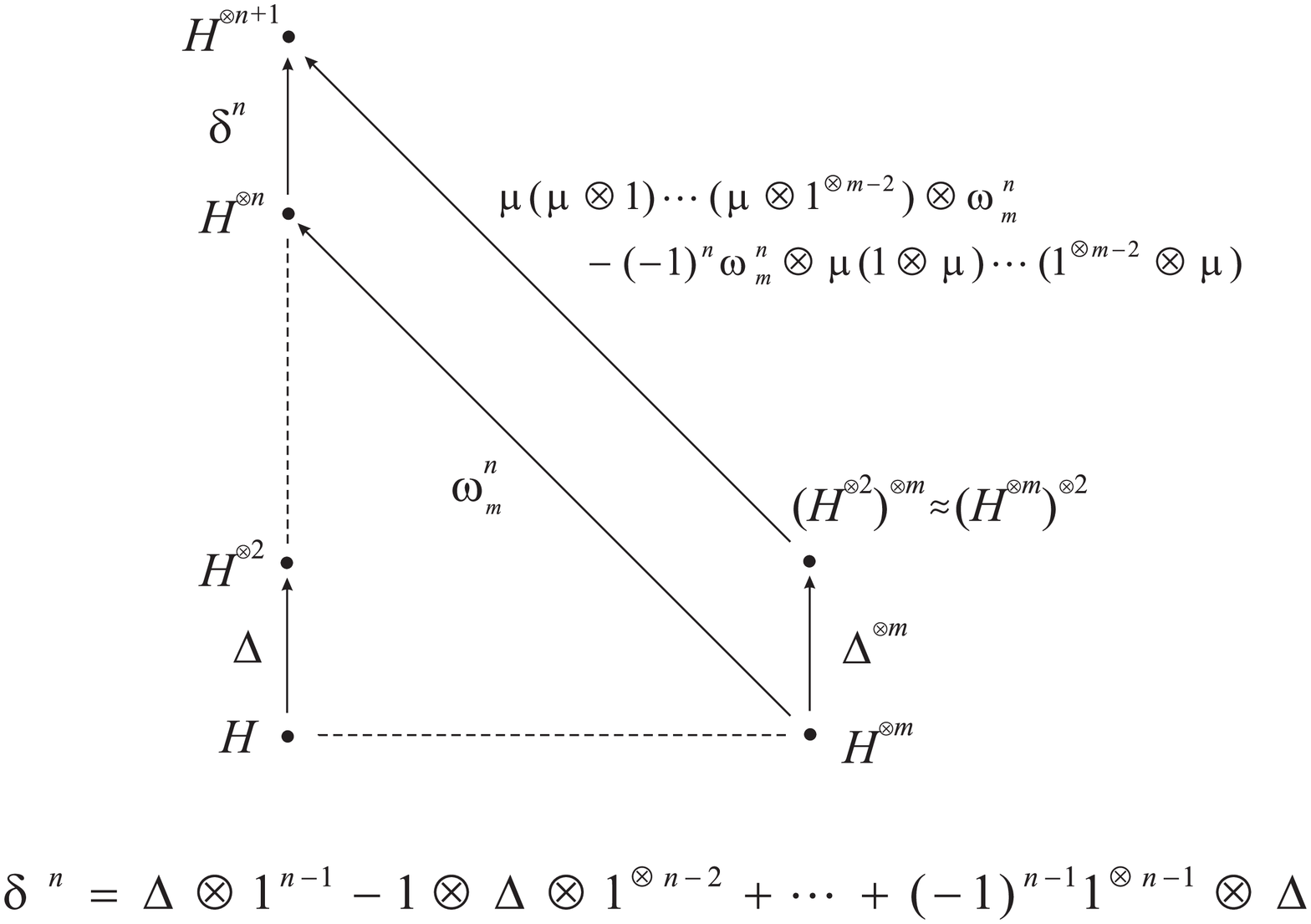}%
\\
Figure 4. Structure relation 1.
\end{center}
Similarly, the homogeneous components of $d_{\omega}\circledcirc d_{\omega}=0$
with $m$ inputs and $n+1$ outputs produce Structure Relation 2:%
\[
\omega_{m}^{n}\partial_{m}=\left(  -1\right)  ^{\left\lfloor \left(
n+1\right)  /2\right\rfloor }\mu^{\otimes n}\sigma_{n,2}\left(  \omega_{m}%
^{n}\otimes f^{n}-\left(  -1\right)  ^{m}f^{n}\otimes\omega_{m}^{n}\right)  .
\]%
\begin{center}
\includegraphics[
trim=0.000000in -0.124936in 0.000000in 0.000000in,
height=2.6161in,
width=4.1122in
]%
{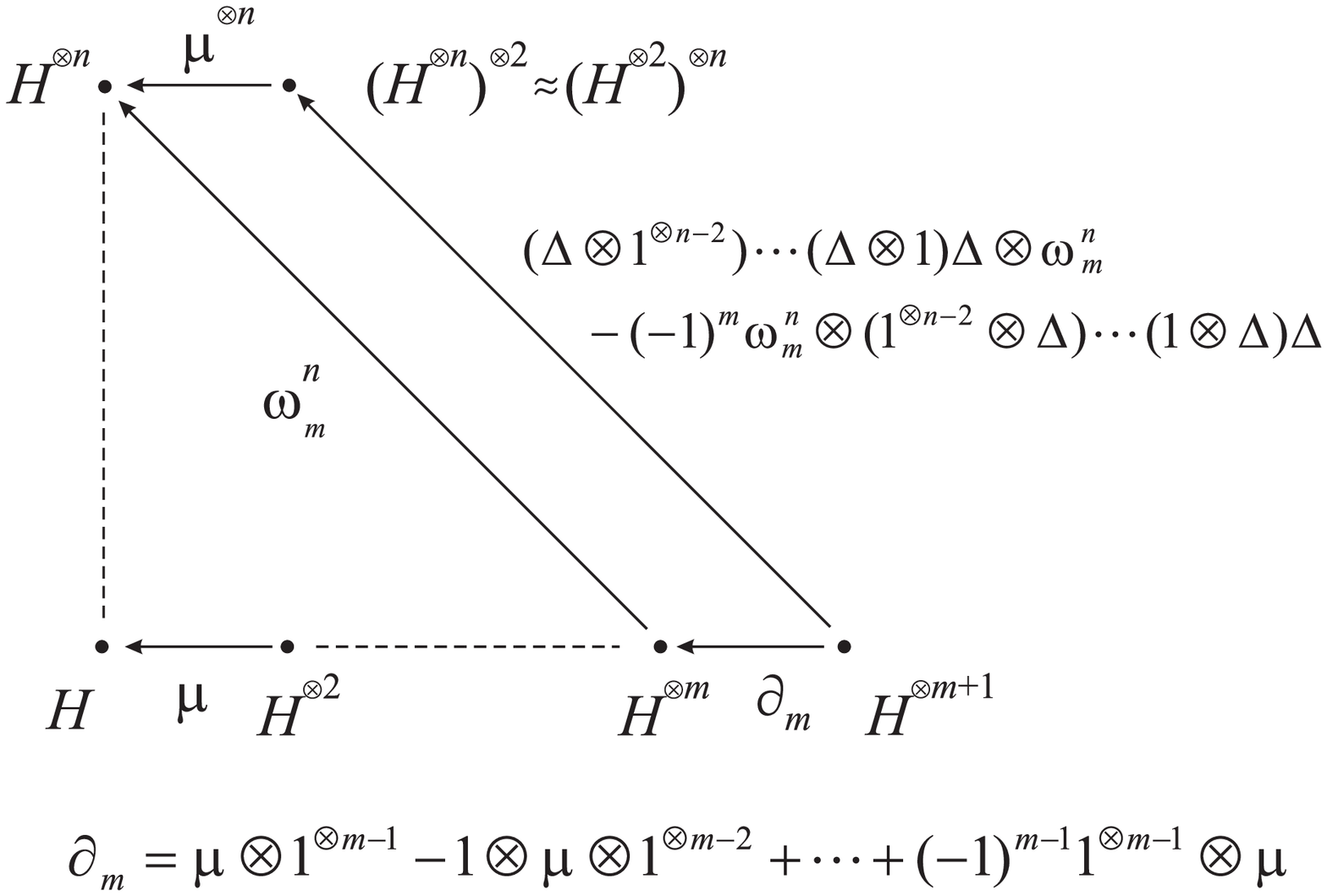}%
\\
Figure 5. Structure relation 2.
\end{center}
\vspace*{0.2in}

Additional relations appear in Definition \ref{defn1} below. Whereas our
examples do not depend on the signs in Relations 3, 4(a-c), and 5(a-c), we
omit the complicated sign formulas. Note that Relation 3 is a classical
$A_{\infty}$-(co)algebra relation when $n=1$ (or $m=1$)$.$

\begin{definition}
\label{defn1}Given $m,n\in\mathbb{N}$ with $m+n\geq4,$ an $A_{\infty}%
$\emph{-bialgebra of type} $\left(  m,n\right)  $ is a $\Lambda$-Hopf algebra
$\left(  H,\Delta,\mu\right)  $ together with an operation $\omega_{m}^{n}\in
Hom^{m+n-3}\left(  H^{\otimes m},H^{\otimes n}\right)  $ such that

\begin{enumerate}
\item[1.] $\delta^{n}\omega_{m}^{n}=\left(  g_{m}\otimes\omega_{m}^{n}-\left(
-1\right)  ^{n}\omega_{m}^{n}\otimes g_{m}\right)  \sigma_{2,m}\Delta^{\otimes
m}$

\item[2.] $\omega_{m}^{n}\partial_{m}=\mu^{\otimes n}\sigma_{n,2}\left(
f^{n}\otimes\omega_{m}^{n}-\left(  -1\right)  ^{m}\omega_{m}^{n}\otimes
f^{n}\right)  $

\item[3.] $\left(  \omega_{m}^{n}\otimes g_{m}^{\otimes n-1}\pm g_{m}%
\otimes\omega_{m}^{n}\otimes g_{m}^{\otimes n-2}\pm\cdots\pm g_{m}^{\otimes
n-1}\otimes\omega_{m}^{n}\right)  \sigma_{n,m}\smallskip\newline%
\hspace*{0.2in}\left.  \left(  \omega_{m}^{n}\otimes\left(  f^{n}\right)
^{\otimes m-1}\pm f^{n}\otimes\omega_{m}^{n}\otimes\left(  f^{n}\right)
^{\otimes m-2}\pm\cdots\pm\left(  f^{n}\right)  ^{\otimes m-1}\otimes
\omega_{m}^{n}\right)  =0\right.  $

\item[4.] If $m=2,$ then

\begin{enumerate}
\item[a.] $\left(  \omega_{2}^{n}\otimes\omega_{2}^{n}\right)  \sigma
_{2,2}\left(  \Delta\otimes\Delta\right)  =0$

\item[b.] $\left(  \omega_{2}^{n}\right)  ^{\otimes n}\sigma_{n,2}\left(
f^{n}\otimes\omega_{2}^{n}-\omega_{2}^{n}\otimes f^{n}\right)  =0$

\item[c.] $\left[  \mu\otimes\left(  \omega_{2}^{n}\right)  ^{\otimes n-1}%
\pm\omega_{2}^{n}\otimes\mu\otimes\left(  \omega_{2}^{n}\right)  ^{\otimes
n-2}\right.  \newline\hspace*{0.3in}\left.  \pm\cdots\pm\left(  \omega_{2}%
^{n}\right)  ^{\otimes n-1}\otimes\mu\right]  \sigma_{n,2}\left(  f^{n}%
\otimes\omega_{2}^{n}-\omega_{2}^{n}\otimes f^{n}\right)  =0\newline%
\hspace*{1in}\vdots$

\item[d.] $\left[  \mu^{\otimes n-1}\otimes\omega_{2}^{n}+\mu^{\otimes
n-3}\otimes\omega_{2}^{n}\otimes\mu^{\otimes2}+\cdots\right.  \newline%
\hspace*{0.3in}\left.  -\left(  -1\right)  ^{n}\left(  \mu^{\otimes
n-2}\otimes\omega_{2}^{n}\otimes\mu+\mu^{\otimes n-4}\otimes\omega_{2}%
^{n}\otimes\mu^{\otimes3}+\cdots\right)  \right]  \newline\hspace
*{0.6in}\sigma_{n,2}\left(  f^{n}\otimes\omega_{2}^{n}-\omega_{2}^{n}\otimes
f^{n}\right)  =0$
\end{enumerate}

\item[5.] If $n=2,$ then

\begin{enumerate}
\item[a.] $\left(  \mu\otimes\mu\right)  \sigma_{2,2}\left(  \omega_{m}%
^{2}\otimes\omega_{m}^{2}\right)  =0$

\item[b.] $\left(  g_{m}\otimes\omega_{m}^{2}-\omega_{m}^{2}\otimes
g_{m}\right)  \sigma_{2,m}\left(  \omega_{m}^{2}\right)  ^{\otimes m}=0$

\item[c.] $\left(  g_{m}\otimes\omega_{m}^{2}-\omega_{m}^{2}\otimes
g_{m}\right)  \sigma_{2,m}\left[  \Delta\otimes\left(  \omega_{m}^{2}\right)
^{\otimes m-1}\pm\omega_{m}^{2}\otimes\Delta\otimes\left(  \omega_{m}%
^{2}\right)  ^{\otimes m-2}\right.  \newline\hspace*{0.3in}\left.  \pm
\cdots\pm\left(  \omega_{m}^{2}\right)  ^{\otimes m-1}\otimes\Delta\right]
=0\newline\hspace*{1in}\vdots$

\item[d.] $\left(  g_{m}\otimes\omega_{m}^{2}-\omega_{m}^{2}\otimes
g_{m}\right)  \sigma_{2,m}\left[  \Delta^{\otimes m-1}\otimes\omega_{m}%
^{2}+\Delta^{\otimes m-3}\otimes\omega_{m}^{2}\otimes\Delta^{\otimes2}%
+\cdots\right.  \newline\hspace*{0.3in}-\left(  -1\right)  ^{m}\left.
\Delta^{\otimes m-2}\otimes\omega_{m}^{2}\otimes\Delta+\Delta^{\otimes
m-4}\otimes\omega_{m}^{2}\otimes\Delta^{\otimes3}+\cdots\right]  =0$
\end{enumerate}

\item[6.] If $m=n=2,$ then $\left(  \omega_{2}^{2}\otimes\omega_{2}%
^{2}\right)  \sigma_{2,2}\left(  \omega_{2}^{2}\otimes\omega_{2}^{2}\right)
=0.\vspace*{0.1in}$
\end{enumerate}
\end{definition}

\begin{example}
\label{graphical}When $\left(  m,n\right)  =\left(  2,3\right)  ,$ the
relations in Definition \ref{defn1} are represented graphically as
follows:$\smallskip$
\end{example}

\noindent Relation 1: $\left(  \Delta\otimes\mathbf{1}^{\otimes2}%
-\mathbf{1}\otimes\Delta\otimes\mathbf{1}+\mathbf{1}^{\otimes2}\otimes
\Delta\right)  \omega_{2}^{3}=\left(  \mu\otimes\omega_{2}^{3}+\omega_{2}%
^{3}\otimes\mu\right)  \sigma_{2,2}\left(  \Delta\otimes\Delta\right)
\medskip$

\begin{center}%
\begin{center}
\includegraphics[
height=0.4142in,
width=3.1505in
]%
{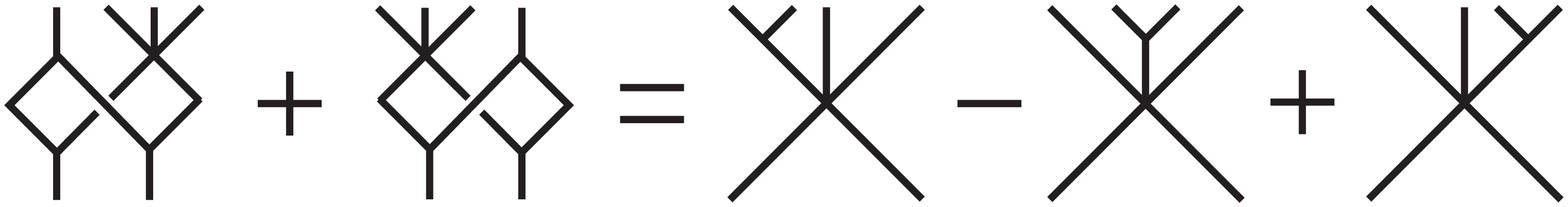}%
\end{center}
$\smallskip$
\end{center}

\noindent Relation 2:\ $\omega_{2}^{3}\left(  \mu\otimes\mathbf{1}%
-\mathbf{1}\otimes\mu\right)  =\mu^{\otimes3}\sigma_{3,2}\left(  \left(
\Delta\otimes\mathbf{1}\right)  \Delta\otimes\omega_{2}^{3}-\omega_{2}%
^{3}\otimes\left(  \mathbf{1}\otimes\Delta\right)  \Delta\right)  .\medskip$

\begin{center}%
\begin{center}
\includegraphics[
height=0.4073in,
width=2.3177in
]%
{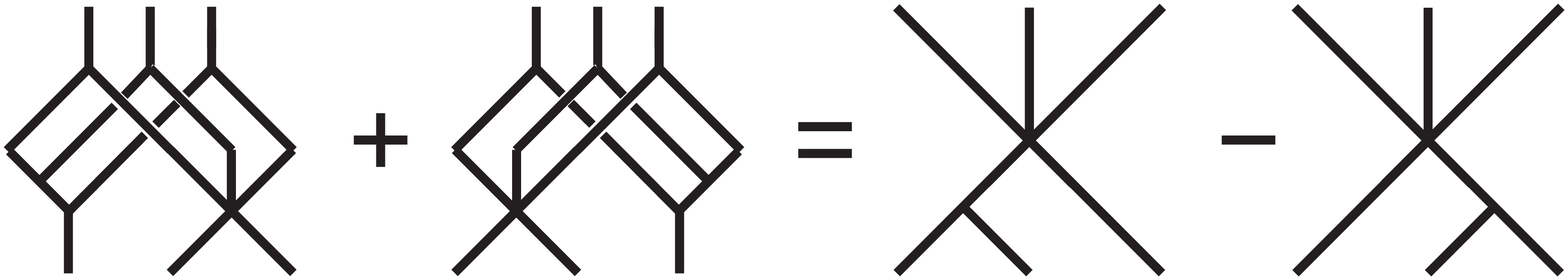}%
\end{center}

\end{center}

\noindent Relation 3 (4d):$\smallskip$

\noindent$\left(  \omega_{2}^{3}\otimes\mu^{\otimes2}+\mu\otimes\omega_{2}%
^{3}\otimes\mu+\mu^{\otimes2}\otimes\omega_{2}^{3}\right)  \sigma_{3,2}\left(
\left(  \Delta\otimes\mathbf{1}\right)  \Delta\otimes\omega_{2}^{3}-\omega
_{2}^{3}\otimes\left(  \mathbf{1}\otimes\Delta\right)  \Delta\right)
=0.\medskip$

\begin{center}%
\begin{center}
\includegraphics[
height=0.48in,
width=4.7556in
]%
{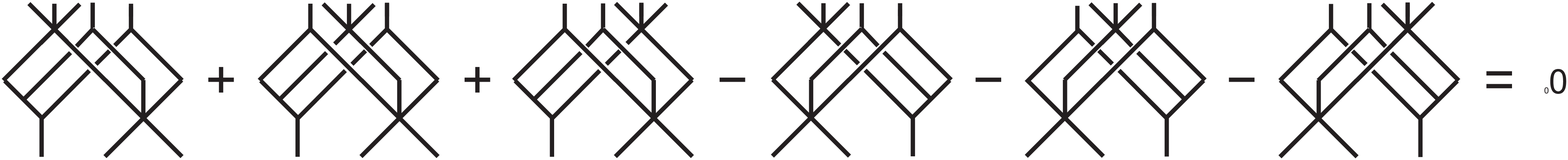}%
\end{center}
$\smallskip$
\end{center}

\begin{example}
\label{ex1}Let $H=\left\langle y\right\rangle $ be a graded $\mathbb{Z}_{2}%
$-module with $|y|=-2$. Let $\mu$ and $\Delta$ be the trivial product and
primitive coproduct on $H$, i.e., $\mu\left(  1|y\right)  =\mu\left(
y|1\right)  =y$ and $\Delta\left(  y\right)  =1|y+y|1$. Then $H$ is a Hopf
algebra. Define an operation $\omega_{2}^{3}:H^{\otimes2}\rightarrow
H^{\otimes3}$ by $\omega_{2}^{3}\left(  y|y\right)  =y|1|1+1|1|y$ and zero
otherwise. Then $\omega_{2}^{3}$ satisfies the structure relations in
Definition \ref{defn1} and $\left(  H,\Delta,\mu,\omega_{2}^{3}\right)  $ is
an $A_{\infty}$-bialgebra of type $\left(  2,3\right)  .$ We verify Relation
1; verification of the other relations is left to the reader.%
\[%
\begin{tabular}
[c]{l}%
$(\mu\otimes\omega_{2}^{3}+\omega_{2}^{3}\otimes\mu)\sigma_{2,2}(\Delta
\otimes\Delta)(y|y)$\\
$\hspace*{0.25in}=(\mu\otimes\omega_{2}^{3}+\omega_{2}^{3}\otimes\mu)\left(
y|y|1|1+y|1|1|y+1|y|y|1+1|1|y|y\right)  $\\
$\hspace*{0.25in}=(y|1|1+1|1|y)\otimes1+1\otimes(y|1|1+1|1|y)$\\
$\hspace*{0.25in}=y|1|1|1+1|1|y|1+1|y|1|1+1|1|1|y$\\
$\hspace*{0.25in}=(\Delta\otimes\mathbf{1}^{\otimes2}+\mathbf{1}\otimes
\Delta\otimes\mathbf{1}+\mathbf{1}^{\otimes2}\otimes\Delta)\omega_{2}%
^{3}(y|y)=\delta^{3}\omega_{2}^{3}(y|y)$.
\end{tabular}
\ \ \ \ \ \ \
\]

\end{example}

Before we present our general family of examples, we determine all possible
dimensions for our generators. Let $R$ be a commutative ring with unity.

\begin{proposition}
\label{prop1}Let $p,q\in\mathbb{N}$, let $\Lambda=E\left(  x\right)  $ be an
exterior $R$-algebra generated by $x$ of degree $p,$ and let $H=\left\langle
y\right\rangle $ be a graded $\Lambda$-module generated by $y$ of degree $q$.
The operation $\omega_{m}^{n}:H^{\otimes m}\rightarrow H^{\otimes n}$ defined
by
\[
\omega_{m}^{n}\left(  y|\cdots|y\right)  =xy|y|\cdots|y
\]
has degree $m+n-3$ if

\begin{enumerate}
\item[\textit{(i)}] $m=n=2,$ $p=1,$ and $q\geq2;$

\item[\textit{(ii)}] $m\geq2,$ $n\geq1,$ $p=2m-3,$ and $q=1;$

\item[\textit{(iii)}] $m\geq3,$ $m+1\leq n\leq3m-p-3,$ $1\leq p\leq2m-4,$ and
$q=\left(  p-m-n+3\right)  /\left(  m-n\right)  .$
\end{enumerate}
\end{proposition}

\begin{proof}
In general, $\left\vert \omega_{m}^{n}\right\vert =m+n-3$ if and only if
\begin{equation}
m\left(  q+1\right)  =n\left(  q-1\right)  +p+3. \label{one}%
\end{equation}
Thus $p,q\geq1$ implies $m\geq2.$\smallskip\newline(i) If $m=n=2,$ $p=1,$ and
$q\geq2,$ equation \ref{one} implies $\left\vert \omega_{2}^{2}\right\vert
=1$.\smallskip\newline(ii) If $m\geq2,$ $n\geq1,$ $p=2m-3,$ and $q=1,$
equation \ref{one} implies that $\left\vert \omega_{m}^{n}\right\vert
=m+n-3$.\smallskip\newline(iii) If $m\geq3,$ $1\leq p\leq2m-4,$ and $q\geq2,$
equation \ref{one} defines a rational function
\[
q\left(  n\right)  =\frac{p-m-n+3}{m-n},
\]
which decreases from $2m-p-2$ to $2$ as $n$ increases from $m+1$ to $3m-p-3$
(see Figure 6 below). Consequently $\left\vert \omega_{m}^{n}\right\vert
=m+n-3$ whenever $m+1\leq n\leq3m-p-3$ and $q\left(  n\right)  \in\mathbb{N}$.
\end{proof}

%

\begin{center}
\includegraphics[
trim=0.000000in 2.818712in 0.000000in 0.403122in,
height=2.7466in,
width=2.7198in
]%
{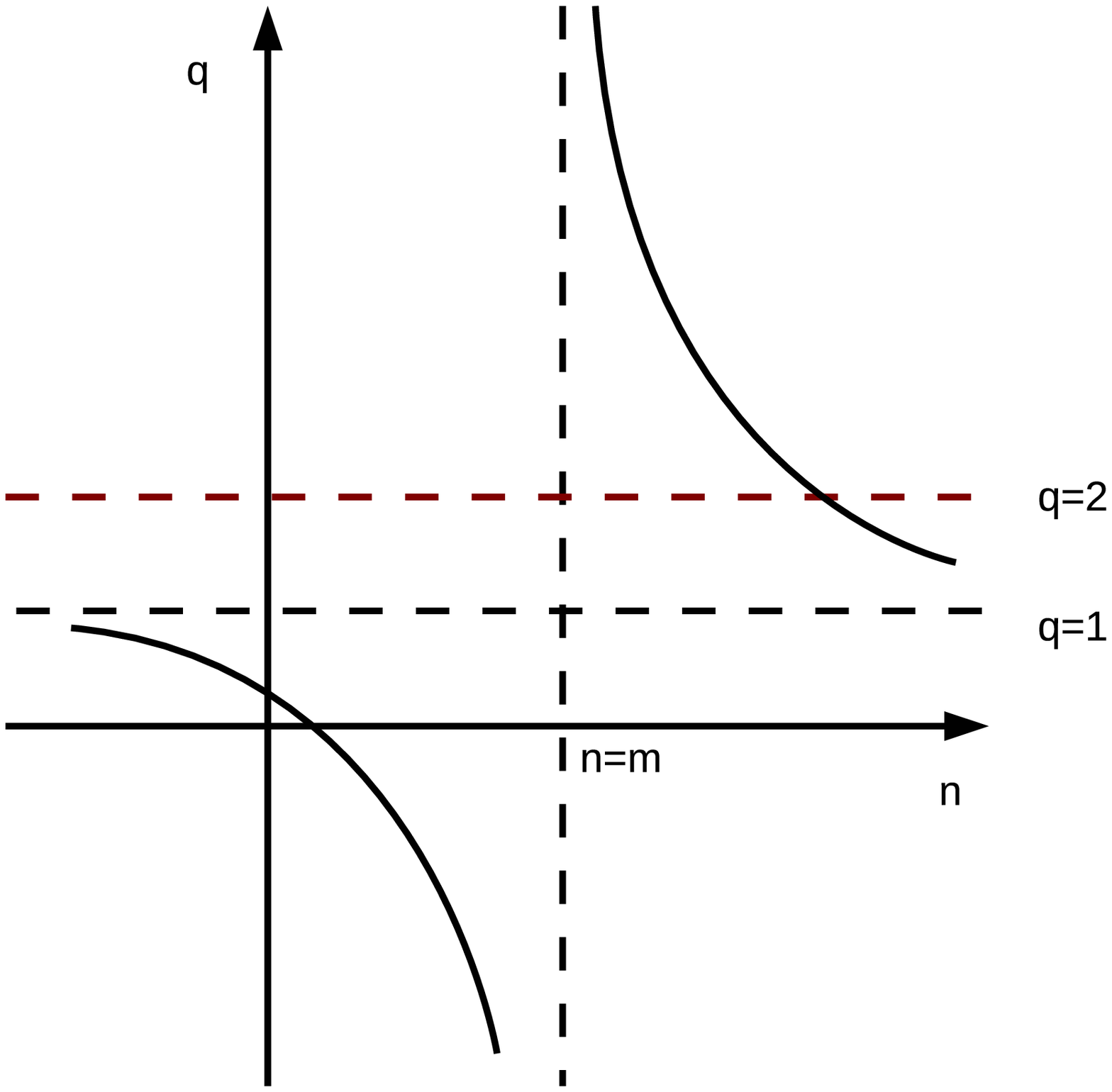}%
\\
Figure 6.
\end{center}

Our main result now follows. Note that when $H$ is a module over an exterior
$R$-algebra $\Lambda=E\left(  x\right)  $, the relation $yx|y=y|xy=\left(
-1\right)  ^{\left\vert x\right\vert \left\vert y\right\vert }x\left(
y|y\right)  $ holds in $TH$ since we tensor over $\Lambda.$ Thus $x^{2}=0$
implies $yx|yx=0.$

\begin{theorem}
\label{thm1}Let $m,$ $n,$ $p,$ and $q$ satisfy the hypotheses of Proposition
\ref{prop1}, and let $\Lambda=E\left(  x\right)  $ be an exterior $R$-algebra
generated by $x$ of degree $p.$ Let $H=E\left(  y\right)  $ be an exterior
$\Lambda$-algebra generated by $y$ of degree $q,$ let $\mu$ and $\Delta$
denote the trivial product and primitive coproduct on $H,$ and define
$\omega_{m}^{n}:H^{\otimes m}\rightarrow H^{\otimes n}$ by $\omega_{m}%
^{n}\left(  y|\cdots|y\right)  =xy|y|\cdots|y.$ Then $\left(  H,\Delta
,\mu,\omega_{m}^{n}\right)  $ is an $A_{\infty}$-bialgebra of type $\left(
m,n\right)  $.
\end{theorem}

\noindent\textit{Proof. }The operation $\omega_{m}^{n}$ has the required
degree by Proposition \ref{prop1}. We must verify that the structure relations
in Definition \ref{defn1} hold:\smallskip

\noindent Relation 1 is non-trivial on $y|\cdots|y$:\smallskip

$(g_{m}\otimes\omega_{m}^{n}-(-1)^{n}\omega_{m}^{n}\otimes g_{m})\sigma
_{2,m}\Delta^{\otimes m}(y|\cdots|y)$\smallskip

$\hspace*{0.3in}=(g_{m}\otimes\omega_{m}^{n}-(-1)^{n}\omega_{m}^{n}\otimes
g_{m})\sigma_{2,m}[(1|y+y|1)|\cdots|(1|y+y|1)]$\smallskip

$\hspace*{0.3in}=g_{m}\left(  1|\cdots|1\right)  \otimes\omega_{m}^{n}\left(
y|\cdots|y\right)  -(-1)^{n}\omega_{m}^{n}\left(  y|\cdots|y\right)  \otimes
g_{m}\left(  1|\cdots|1|\right)  $\smallskip

$\hspace*{0.3in}=1|xy|y|\cdots|y+(-1)^{n+1}xy|y|\cdots|y|1=\delta
^{n}(xy|y|\cdots|y)=\delta^{n}\omega(y|\cdots|y).$\smallskip

\noindent Relation 2 is non-trivial on $y|y|\cdots|y|1$ and $1|y|\cdots|y,$
for example,\smallskip

$\mu^{\otimes n}\sigma_{n,2}(f^{n}\otimes\omega_{m}^{n}-(-1)^{m}\omega_{m}%
^{n}\otimes f^{n})(y|y|\cdots|y|1)$\smallskip

$\hspace*{0.3in}=(-1)^{m+1}\mu^{\otimes n}\sigma_{n,2}(xy|y|\cdots
|y|1|\cdots|1)$\smallskip

$\hspace*{0.3in}=(-1)^{m+1}xy|y|\cdots|y=\omega_{m}^{n}\partial_{m}%
(y|y|\cdots|y|1).$\smallskip

\noindent Relations 3, 4, 5 and 6 hold trivially since each term is a
composition of two tensors, each with a factor of $\omega_{m}^{n}$ whose
non-vanishing values have coefficient $x$. $\ \square\medskip$

Since odd dimensional spheres are rational $H$-spaces, the module
$H=$\linebreak$H_{\ast}\left(  S^{1};H_{\ast}\left(  S^{2m-3};\mathbb{Q}%
\right)  \right)  $ admits a non-trivial $A_{\infty}$-bialgebra structure of
type $\left(  m,n\right)  .$ Thus we ask\textbf{: }Does there exist an
$H$-space $X,$ a coefficient ring $\Lambda,$ and an induced operation
$\omega_{m}^{n}$ such that $\left(  H_{\ast}\left(  X;\Lambda\right)
,\Delta,\mu,\omega_{m}^{n}\right)  $ is an $A_{\infty}$-bialgebra?\medskip

\noindent\textbf{Acknowledgement. }We wish to thank Jim Stasheff for reading
an early draft of this paper and offering many helpful suggestions.

\end{document}